\documentclass[article]{amsart}

\usepackage{amsmath}
\usepackage{amsfonts}
\usepackage{amsmath,mathtools}
\usepackage{amssymb}
\usepackage{url}
\usepackage{amscd}
\usepackage{bbm}
\usepackage{amsthm}
\usepackage{tikz}
\usepackage{graphicx}
\usetikzlibrary{matrix}
\usepackage{caption}
\usepackage{mathrsfs}
\usepackage [english]{babel}
\usepackage [autostyle, english = american]{csquotes}
\MakeOuterQuote{"}
\numberwithin{equation}{section}
\def\bysame{\leavevmode\hbox to3em{\hrulefill}\thinspace}

\theoremstyle{plain}
\newtheorem{theorem}{Theorem}[section]
\newtheorem*{theorem*}{Theorem}

\newtheorem{corollary}[theorem]{Corollary}
\newtheorem{lemma}{Lemma}[section]

\newtheorem*{conjecture*}{Conjecture}
\newtheorem*{lemma*}{Lemma}

\theoremstyle{definition}
\newtheorem{definition}{Definition}[section]
\newtheorem*{definition*}{Definition}

\theoremstyle{remark}
\newtheorem{remark}{Remark}[section]
\newtheorem{example}{Example}[section]
\newtheorem{question}{Question}

\begin{document}

\date{} 
\title[A hierarchy for closed $n$-cell-complements]{A Hierarchy for Closed $n$-Cell-Complements}
\author{Robert J. Daverman and Shijie Gu}

\address{Department of Mathematics\\
University of Tennessee, Knoxville\\
TN, 37996}
\email{daverman@math.utk.edu}
\address{Department of Mathematical Sciences\\
University of Wisconsin, Milwaukee\\
WI, 53211}
\email{shijiegu@uwm.edu}


\begin{abstract}
Let $C$ and $D$ be a pair of crumpled $n$-cubes and 
$h$ a homeomorphism of $\text{Bd }C$ to $\text{Bd }D$ for which
there exists a map $f_h: C\to D$ such that $f_h|\text{Bd }C =h$ and
$f_{h}^{-1}(\text{Bd }D)=\text{Bd }C$. In our view the presence of such a triple   $(C,D,h)$ suggests that  $C$  is "at least as wild as" $D$.  The collection $\mathscr{W}_n$ of all such triples  is the subject of this paper. If $(C,D,h)\in \mathscr{W}_n$ but there is no homeomorphism such 
that $D$ is at least as wild as $C$, we say $C$ is "strictly wilder than" $D$. The latter concept imposes a partial order on the collection of
crumpled $n$-cubes.  Here we study features of these wildness comparisons, and we present certain attributes of crumpled cubes that are preserved by the maps arising when $(C,D,h) \in \mathscr{W}_n$.  The effort can be viewed as an initial way of classifying the wildness of crumpled cubes.
 
\end{abstract}
\maketitle


\newcommand\sfrac[2]{{#1/#2}}

\newcommand\cont{\operatorname{cont}}
\newcommand\diff{\operatorname{diff}}
\section{Introduction} 
The existence of wildly embedded spheres in the $n$-sphere $S^n$ has been
recognized since the 1920's, with the publication of the famous Alexander
Horned Sphere \cite{Alexander 1} and a related 2-sphere wildly embedded in $S^3$ presented by Antoine \cite{Antoine 1}. Later in the 20th century, there was an extensive study of
conditions under which an $(n-1)$-sphere in $S^n$ is locally flat and, hence,
standardly embedded.  Little has been done, however, to classify or
organize the rich variety of wildly embedded objects.  This paper strives
to initiate that organizational effort.

To that end, we consider triples $(C,D,h)$ consisting of a pair of crumpled $n$-cubes $C$
and $D$ and a homeomorphism $h$ from the boundary of the first to the boundary
of the second, and we name the subcollection $\mathscr{W}_n$ consisting of all such triples $(C,D,h)$ for
which there exists a map $f_h:C \to D$ extending $h$ such that $f_h^{-1}(\text{Bd }D)
= \text{Bd }C$.  We call $f_h$ a map associated with the triple.  Given $(C,D,h) \in \mathscr{W}_n$, we think of $C$ 
as being at least as wild as $D$. Of course, this wildness measure $\mathscr{W}_n$ depends heavily on the
homeomorphism $h$, so we regard $C$ as being at least as wild as $D$ provided
there is some homeomorphism $h$ for which $(C,D,h) \in \mathscr{W}_n$.

Several results established here offer justification for this measure as a
rating of wildness. For instance, when $(C,D,h) \in \mathscr{W}_n$, $f_h$ must induce an epimorphism
$\pi_1(\text{Int }C) \to \pi_1(\text{Int }D)$; any homotopy taming set for $C$ must be sent
to a homotopy taming set for $D$; as a consequence, the wild set of $D$ 
(that is, the set of points of  at which Bd $D$ fails to be locally collared in $D$) 
must lie in the image under $h$ of the wild
set of $C$.

Given $C$, we describe a standard flattening away from a closed subset $X$ of $\text{Bd }C$ that produces
a new crumpled $n$-cube $C_X$ and a homeomorphism $h_X:\text{Bd } C \to \text{Bd } C_X$ such that
$(C,C_X,h_X) \in \mathscr{W}_n$ and $\text{Bd }C_X$ is locally flat in $C_X$ at all points
of $h_X(\text{Bd }C - X) = h_X(\text{Bd }C) - X$. Moreover, when $X$ is the closure of an open subset of the
wild set for $C$, then $h_X(X) = X$ equals the wild set for $C_X$.  The standard
flattening technique furnishes an efficient method for presenting unusual
examples.

We also introduce a notion of "strictly wilder than", saying that a
crumpled $n$-cube $C$ is strictly wilder than another crumpled cube $D$ if there
exists a homeomorphism $h$ such that $(C,D,h) \in \mathscr{W}_n$ but there is no
homeomorphism $H:\text{Bd }D \to \text{Bd }C$ such that $(D,C,H) \in \mathscr{W}_n$.  We study
this partial order briefly in Section 5.  If the crumpled $n$-cube $C$
contains a spot at which its boundary is locally flat, or if $\text{Bd }C$ has
finitely generated fundamental group, we show that $C$ cannot be a maximal
element in this partial order; we suspect there are no maximal elements
whatsoever, but have been unable to confirm the suspicion. The preservation
of "at least as wild as" under different operations such as suspension and spin operations
is discussed in Section 4. In Section 6, the sewing space of crumpled cubes is shown to have some nice
properties whenever the "at least as wild as" condition prevails. 

Maps like $f_h$ have been used by Wang \cite{Wang 1} and others to impose a partial order on knots in $S^3$.


\section{Definitions and Notation}
The symbol $\text{Cl }A$ is used to denote the closure of $A$; the boundary and interior of $A$ are denoted as 
$\text{Bd }A$ and $\text{Int }A$; the symbol $\mathbbm{1}$ is the identity map.
\begin{definition} 
A $crumpled$ $n$-$cube$ $C$ is a space homeomorphic to the union of an ($n-1$)-sphere $\Sigma$ in $S^n$ and one of its 
complementary domains.  The sphere $\Sigma$ is the \textit{boundary of }$C$,
 written Bd $C$, and $C - \Sigma$ is the \textit{interior of }$C$, written Int $C$.
\end{definition}

\begin{definition} 
A \textit{closed} $n$-\textit{cell-complement} is a crumpled $n$-cube $C$ embedded  in $S^n$ so that $S^n - \text{Int }C$ is an $n$-cell.
\end{definition}

Every crumpled $n$-cube admits such an embedding \cite{Daverman 2}\cite{Daverman 3}\cite{Hosay 1}\cite{Lininger 1}. 
This concept arises because, when dealing with the possible wildness in $S^n$  of a compact subset of $\text{Bd }C$, it is useful to treat $C$ as a closed $n$-cell-complement, in order to preclude any wildness complications arising from the other crumpled cube, $S^n - \text{Int }C$.

\begin{definition} 
A subset $T$  of the boundary of a crumpled cube $C$ is a
\textit{homotopy taming set} for $C$ if every map 
$m:I^2 \to C$  can be approximated by a map  $m':I^2 \to C$  such that
$m'(I^2) \subset T \cup \text{Int }C$.
\end{definition}

Every crumpled $n$-cube has a 1-dimensional homotopy taming 
set \cite{Daverman 7}.
All crumpled 3-cubes have 0-dimensional homotopy taming sets; it is unknown whether the same is true for crumpled $n$-cubes when $n > 3$.

Crumpled cubes with particularly nice homotopy taming sets are referred to as follows:

\begin{definition} 
A crumpled $n$-cube $C$ is \textit{Type 1} if there exists a 0-dimensional
homotopy taming set $T$ in $\text{Bd }C$ such 
that $T$ is a countable union of Cantor sets that are tame relative to $\text{Bd }C$.
\end{definition}

\begin{definition} 
The \textit{inflation} of a crumpled $n$-cube $C$ is 
$$\text{Infl}(C,d)=\{\langle c,t \rangle\ \in C\times \mathbb{R}^1|c\in C \text{ and }
|t|\leq d(c)\},$$
where $d: C \rightarrow [0,1]$ is a map such that $d^{-1}(0)=\text{Bd }C$ 
\cite[P.270]{Daverman 1}.  Neither the homeomorphism type nor the embedding type of 
Infl$(C,d)$ depends on the choice of map $d$, so ordinarily we suppress reference to $d$. 
\end{definition}

\begin{definition} 
Let $C$ be a crumpled $n$-cube.  A point  $p \in \text{Bd }C$ is a \textit{piercing point} of $C$ if there 
exists an embedding $\xi$ of $C$ in the $n$-sphere $S^n$ such that $\xi(\text{Bd }C)$ can be pierced with a tame arc at 
$\xi(p)$.
\end{definition}

All boundary points of crumpled $n$-cubes are piercing points when $n > 3$.  McMillan 
\cite{McMillan 2} has shown that boundary points $p$ of crumpled 3-cubes $C$ are piercing points of $C$ if and only if $C - p$ is locally simply connected at $p$.

\begin{definition} 
A proper map $f: M \rightarrow \tilde{M}$ between connected, orientable $n$-manifolds has \textit{degree one} if $f$ induces an
isomorphisms of (cohomology groups with compact supports) $H^{n}_{c}(\tilde{M}) \to  H^{n}_{c}(M)$.
\end{definition}


\section{Some Basic Properties fof the Collection $\mathscr{W}_n$}
The fundamental aim here is the attempt to measure or compare the wildness of two given crumpled $n$-cubes using $\mathscr{W}_n$. 
Two obvious but basic features are worth noting: 
\begin{itemize}
\item[(1)] $(C,C, \mathbbm{1})\in \mathscr{W}_n$,
\item[(2)] $(C,C',h)$  and $(C',C'',h')$ in $\mathscr{W}_n$ implies $(C,C'',h'h)$ is in $\mathscr{W}_n$.
\end{itemize}

Loosely speaking, we think of $C$ as being wilder than $D$ if there exists 
$(C,D,h) \in \mathscr{W}_n$. This comparison, however, raises the following 
question: for $(C,D,h)$ in $\mathscr{W}_n$, does there ever exist 
$(D,C,H) \in \mathscr{W}_n$?  The basic feature that $(C,C,\mathbbm{1})\in \mathscr{W}_n$ supplies an affirmative answer and indicates that this 
"wilder than" language is misleading.  Accordingly, we phrase the concept more conservatively  as follows: 
 
\begin{definition} 
$C$ is \textit{at least as wild as} $D$ if and only if there exists a homeomorphism $h$ such that $(C,D,h)\in 
\mathscr{W}_n$.
\end{definition}

\begin{theorem} 
For every crumpled $n$-cube $C$ and every homeomorphism $h$ from $\mathrm{Bd }C$ to $\mathrm{Bd }B^n$,
$(C,B^n,h)\in \mathscr{W}_n$.
\end{theorem}
\begin{proof}
Since $B^n$ is an Absolute Retract, the homeomorphism  $h$  
extends to a map 
$f_h:C\rightarrow B^n$. Think of $B^n$ as the unit ball. Restrict the metric on $C$ so $C$ has diameter $\leq 1$. Treat 
$f_h(x)$ as a vector from the origin through the image point $f_h(x)$. Modify $f_h$ by sending  any  $x$ in $C$ to the 
vector $\big(1 -\text{dist}(x,\text{Bd }C)\big) \cdot f_h(x)$ (i.e. scalar product).  Now $f_h^{-1}(\text{Bd }B^n) = \text{Bd }C.$
\end{proof}

\begin{theorem}
Suppose $C$ and $D$ are crumpled $n$-cubes such that $\mathrm{Bd}$ $C$ and $\mathrm{Bd}$ $D$ have closed  neighborhoods
$U_C, U_D$ in $C,D$, respectively, that are homeomorphic via $H: U_C \to U_D$, and suppose $\pi_1(\mathrm{Int}$ $D)=1$. Then $(C,D, H|\mathrm{Bd}$ $C)\in \mathscr{W}_n$.
\end{theorem}
\begin{proof}
Here $\text{Int }D$ is homologically and homotopically trivial, implying that $H$ extends to a map $f_H:C \to D$ with $f_H(C - U_C) \subset \text{Int }D.$ Clearly $f_H$ assures that $(C,D,H|\text{Bd }C) \in \mathscr{W}_n$.
\end{proof}

Next we outline an example showing the existence of $(C,C,h) \in \mathscr{W}_n$ where no associated map $f_h:C \to C$ can be a homeomorphism. It is meant to suggest that a reflexivity aspect of the "at least as wild as" relation occasionally holds for complicated reasons.  In the next section we will present an example showing that the relation actually fails to be asymmetric and hence does not determine a partial order on the collection of crumpled $n$-cubes. 

\begin{example} 
A triple $(C,C,h) \in \mathscr{W}_3$ such that every associated map $f_h:C \to C$ induces a homomorphism $\pi_1(\text{Int }C) \to \pi_1(\text{Int }C)$ with nontrivial kernel. Consider the closed 3-cell-complement $C$ bounded by Alexander's horned sphere. Wipe out the wildness in one of the two primary horns, but leave 
the other horn unchanged.  The result is a new crumpled 3-cube $D$, simpler than the original in that some of the 
wildness has been eliminated, but nevertheless (by inspection) embedded exactly like the original.  We continue to differentiate the two using  different names, despite the fact that $C$ and $D$ are equivalent. There is a homeomorphism $h$ from 
$\text{Bd } C$ to $\text{Bd } D$  sending one of the primary horns of $C$ onto the wild horn of $D$ (the entire wild part 
of $D$) and sending the other primary horn of $C$ into the flattened part of $D$. It takes a little work to show that
$h$ extends to the appropriate kind of map from $C$ to $D$, but that is quite like showing that $(C,B^n,h) \in 
\mathscr{W}_n$. Let $J \subset \text{Bd }C$ be a simple closed curve separating the two horns of $C$, and note that any associated map $f_h$ must send loops in $\text{Int }C$ near $J$ to homotopically inessential loops in $\text{Int }D$.  
\end{example}

\begin{lemma} 
Suppose $(C,D,h) \in \mathscr{W}_n$, $f_h:C \to D$ is a map
associated with $h$, $W$ is a connected open subset of $D$  such that $W \cap \mathrm{Bd}$ $D$ is connected, 
$Y$ is the component of $f_h^{-1}(W)$ containing
$f_h^{-1}(W \cap \mathrm{Bd}$ $D) = h^{-1}(W \cap \mathrm{Bd}$ $D)$.  Then $f_h$ induces an
epimorphism $\pi_1(Y \cap \mathrm{Int}$ $C) \to \pi_1(W \cap \mathrm{Int}$ $D)$.
\end{lemma}
\begin{proof}
Treat $C$ and $D$ as closed $n$-cell-complements. 
Extend $f_h$ to a proper map $F_h:S^n \to
S^n$ that restricts to a homeomorphism between $S^n - \text{Int }C$ and $S^n - \text{Int }D$.  Being a homemorphism 
over some open subset, $F_h$ must have geometric degree 1.  Then by \cite{Epstein 1}
$$f_h|Y \cap \text{Int }C = F_h|Y \cap \text{Int }C: Y \cap \text{Int }C \to W \cap \text{Int }D$$
also has degree 1, which implies that it induces an epimorphism of
fundamental groups.
\end{proof}

\begin{corollary} 
If $(C,D,h) \in \mathscr{W}_n$  and  $f_h:C\rightarrow D$ is an associated map
extending $h$  with $f_h(\mathrm{Int}$ $C) \subset \mathrm{Int}$ $D$, then
$f_h$ induces an epimorphism  of $\pi_1(\mathrm{Int}$ $C)$ to $\pi_1(\mathrm{Int}$ $D)$.
\end{corollary}
\begin{proof}
Apply Lemma 3.1 with $W = D$ and $Y = C$.
\end{proof}

\begin{corollary} 
If $(C, D, h)\in \mathscr{W}_n$ and $\mathrm{Int}$ $C$ is an open 
$n$-cell, then $\mathrm{Int}$ $D$ is an open $n$-cell.
\end{corollary}
\begin{proof} 
According to  \cite{McMillan 1}, $\text{Int }D$ is an open $n$-cell if (and only if)  $\text{Int }D$ is simply connected at infinity; in other words, given one neighborhood  $U$ of Bd $D$ in $D$, 
one must be able to produce a smaller neighborhood $V$ of $\text{Bd }D$ there such that each loop in $V \cap \text{Int }D$ is null-homotopic in $U \cap \text{Int }D$.
Pull back to $C$.  First, find a connected neighborhood $V'$ of Bd $C$ in $C$ such that, not only is 
$V' \subset f_h^{-1}(U)$ but also all loops in $V' \cap \text{Int }C$
are null-homotopic in $f_h^{-1}(U) \cap \text{Int }C$.
Next, locate a connected neighborhood  $V$ of Bd $D$ in $D$ with $f_h^{-1}(V) \subset V'$.  By Lemma 3.1 every loop $\gamma$  in $V \cap \text{Int }D$ is homotopic there to the image of a loop $\gamma'$ 
in $f_h^{-1}(V) \cap \text{Int }C \subset V' \cap \text{Int }C$, 
and $\gamma'$ in turn is null-homotopic 
in $f_h^{-1}(U) \cap \text{Int }C$. Finally, apply $f_h$ to see that 
$\gamma$ itself is null-homotopic in $U \cap \text{Int }D$. 
\end{proof}

\begin{theorem} 
If $(C,D,h) \in \mathscr{W}_n$ and $T$ is a homotopy taming set for $C$, then $h(T)$ is a homotopy taming set for $D$.
\end{theorem} 
\begin{proof} 
Consider any map $\phi:I^2 \to D$ and any $\epsilon > 0$.  By Lemma 3.1, for any
$x \in \text{Bd }D$ there exist a small connected neighborhood $N_x$ of $x\in D$ and a
small connected neighborhood $M_x$ of $h^{-1}(x)$ in $f_h^{-1}(N_x)$ such that
$f_h$  induces an epimorphism of $\pi_1(M_x \cap \text{Int }C) \to \pi_1(N_x \cap
\text{Int }D)$.  Do this so every loop in an $M_x$ contracts in a subset of $C$ whose
image under $f_h$  in an $(\epsilon/2)$-subset of $D$.  Note that if $L$ is a loop
in an $M_x \cap \text{Int }C$, then its image under $f_h$ contracts in an
$(\epsilon/2)$-subset of $h(T) \cup \text{Int }D$.
Produce a (Lebesgue number) $\delta \in (0,\epsilon/2)$ such that any
$\delta$-subset of $D$ within $\delta$ of $\text{Bd }D$ lies in some $M_x$.  

We define a new map $\phi':I^2 \to D$ such that $\phi'$ is $\epsilon$-close to
$\phi$ and $\phi'(I^2) \subset h(T) \cup \text{Int }D$.  First, impose a
triangulation $\mathscr{T}$ of  $I^2$ with mesh so fine that the diameter of
each $\phi(\Delta)$, where $\Delta$ denotes a 2-simplex of $\mathscr{T}$,  is less than $\delta$. Next, approximate $\phi$ by a
map (still called $\phi$), so that the image of the 1-skeleton of $\mathscr{T}$
avoids $\text{Bd }D$; this can be done without affecting any of the size
controls achieved to this point. Then note that, for those 2-simplexes $\Delta \in
\mathscr{T}$ such that $\phi(\Delta)$ meets $\text{Bd }D$, $\phi(\partial \Delta)$ is
homotopic in some $N_x$ to the image under $f_h$ of a loop $L \subset  M_x$, and
$f_h|L$ bounds a singular disk in an $(\epsilon/2)$-subset of $h(T) \cup \text{Int }D$.  
Of course, $\phi'$ and $\phi$ may as well agree on those $\Delta \in \mathscr{T}$
such that $\phi(\Delta) \cap \text{Bd }D =\emptyset$.
\end{proof}

The following corollaries apply when $(C,D,h) \in \mathscr{W}_n$.
\begin{corollary} 
If $C$ is Type 1, so is $D$.
\end{corollary}

\begin{corollary} 
If $C$ has the Disjoint Disks Property, so does $D$.
\end{corollary}

\begin{corollary} 
If $C$ has a $0$-dimensional homotopy taming set, so does $D$;
\end{corollary}

\begin{remark}
If  $T$  is any homotopy taming set for the crumpled cube $C$, then the wild set
of $C$  is contained in the closure of  $T$.  Essentially by definition of 
homotopy taming set, $\text{Int }C$ is locally simply connected at all points
of $\text{Bd }C - T$, which assures local flatness there 
\cite{Bing 2} \cite{FreedmanQuinn} \cite[\S 7.6]{DavermanVenema}.
\end{remark}

\begin{corollary} 
Suppose $W_C$ and $W_D$ are the wild sets in $\mathrm{Bd}$ $C$ and $\mathrm{Bd}$ $D$ respectively. Then $W_D\subset 
h(W_C)$. 
\end{corollary}
\begin{proof}
Take any homotopy taming set $T$ for $C$.  Note that singular disks in $C$ can be
adjusted, fixing points that are sent to $W_C$, while moving the image off
$\text{Bd }C - W_C$. In other words, $T \cap W_C$ is another homotopy taming set for
$C$.  Since then $h(T \cap W_C)$ is a homotopy taming set for $D$, we have $W_D \subset \text{Cl }h(T \cap W_C) \subset \text{Cl }h(T) \cap h(W_C) \subset h(W_C)$.
\end{proof}

\begin{remark}
Even when $C$ and $D$ are locally flat modulo wild sets $W_C$ and $W_D$ that are 
tame in space, and there is a homeomorphism $h:\text{Bd }C \to \text{Bd } D$  with $h(W_C) = W_D$, 
one cannot infer that 
$C$ is at least as wild as $D$. To see why not,
we consider a pair of different 
crumpled 3-cubes each is locally flat modulo two points, the first might have simply connected interior and second might be non-simply connected. For higher 
dimensional cases, spins or suspensions can be applied to obtain different crumpled $n$-cubes each locally flat modulo 
an $(n-3)$-cell or a pair of $(n-3)$-spheres that are tame in $S^n$.
\end{remark}

\begin{corollary} 
If $(C,D,h)$ and $(D,C,h^{-1})$ both belong to $\mathscr{W}_n$, then $h$ sends the wild set of $C$ onto the wild set 
of $D$.
\end{corollary}

\begin{theorem} 
Suppose $C,D$ are closed n-cell-complements, 
$(C,D,h) \in \mathscr{W}_n$ and $X$ is a compact subset of $\mathrm{Bd}$ $C$ with $\mathrm{dim}$ $X<
n-2$ and $X$ $1$-LCC embedded in $S^n$.  Then $h(X)$ is $1$-LCC embedded in $S^n$.
\end{theorem}
\begin{proof}
Extend $f_h: C \to D$ to a map $F_h:S^n \to S^n$,
with $F_h|S^n - C:S^n - C \to S^n - D$ a homeomorphism.

Let $V$ be a neighborhood of $h(x)\in h(X)$.  Find a smaller connected
neighborhood $W$ of $h(x)$
with $W \cap \text{Bd }D$ connected, such that every loop in $F_h^{-1}(W)$ contracts
in $F_h^{-1}(V)$.
Let $Y$ denote the component of $F_h^{-1}(W)$ containing $x$.  Note that, due to
the dimension restriction, $X$ does not separate $Y$.  By the argument given
for Lemma 3.1, $F_h$ induces an epimorphism $\pi_1(Y - X) \to \pi_1\big(W -
h(X)\big)$.

Consider any loop $\alpha$ in $W - h(X)$.  It is the image there of a loop
$\alpha'$ from $Y - X$.  By design, $\alpha$ is null-homotopic in $F_h^{-1}(V)$;
better yet, since $X$ is 1-LCC embedded, $\alpha'$ is null-homotopic in
$F_h^{-1}(V) - X$.  Application of $F_h$ demonstrates that $\alpha$ is
null-homotopic in $V - h(X)$.

\end{proof}

\begin{corollary} 
If $C$ and $D$ are closed  $n$-cell-complements,
$(C,D,h) \in \mathscr{W}_n$ and
$\mathrm{Bd }$ $C$ is locally flat modulo a Cantor set tamely embedded in $S^n$, where $n\geq 5$,  then $\mathrm{Bd}$ $D$ is locally flat modulo 
a Cantor set tamely embedded in $S^n$.
\end{corollary}

\begin{corollary} 
Let $C$ and $D$  be closed $n$-cell-complements.
If  $(C,D,h) \in \mathscr{W}_n$  and Bd $C$ is locally flat modulo a codimension $3$ subset $W_C$ of $\mathrm{Bd}$ $C$  that is embedded in space as a tame polyhedron, then  $D$ is also locally flat modulo a tame subset.
\end{corollary}

\begin{corollary} 
No closed $n$-cell-complement that is locally flat modulo a tame subset of codimension $3$ or greater can be at least as wild as one which is locally flat modulo a wild set. 
\end{corollary}

\begin{theorem} 
If $(C,D,h)\in \mathscr{W}_3$ and $p$ is a piercing point of $C$, then $h(p)$ is a piercing point of $D$.
\end{theorem}
\begin{proof}
By \cite{McMillan 2} a point $x$ in the boundary of a crumpled cube
$C$  is a piercing point if and only if $C$ has a homotopy
taming set $T$ such that $x \notin T$.  Consequently, the existence
of such a $T$  with $p \notin T$ implies $h(p) \notin h(T)$,
which in turn implies that $h(p)$ is a piercing point of $D$.
\end{proof}

\begin{definition} 
The boundary $\Sigma$ of a crumpled $n$-cube $C$ can be \textit{carefully almost approximated from} Int $C$  provided that, for each $\epsilon > 0$  there exists a locally flat embedding $\theta$ of $\Sigma$ in $S^n$ within $\epsilon$ of the inclusion $\Sigma \to S^n$ such that each component of $\theta(\Sigma) - \text{Int }C$ has diameter less than $\epsilon$ and $\Sigma \cap \theta(\Sigma)$ is covered by the interiors of a finite collection of pairwise disjoint ($n-1$)-cells in $\Sigma$, each of diameter less than $\epsilon$.  
\end{definition}

\begin{theorem} 
Suppose $C,D$ are crumpled $n$-cubes,  
$\mathrm{Bd }$ $C$ can be carefully almost approximated from 
$\mathrm{Int}$ $C$, and
$(C,D,h) \in \mathscr{W}_n$. Suppose also $m:B^2 \to 
\mathrm{Bd }$ $D$ is a map and $\delta$ is a positive number.  Then there exists a map $m':B^2 \to  D$ such that 
$\rho(m',m)<\delta$, $m'|\partial B^2=m|\partial B^2$ and 
$$m'(B^2) \cap \text{Bd } D \subset N_\delta\big(m(\partial B^2)\big).$$
\end{theorem}
\begin{proof}
By Lemma 5.2 of \cite{Daverman 4}, the same conclusion holds for the map 
$h^{-1}m:B^2 \to \text{Bd }C$. The properties in $C$ readily transfer to $D$ via $f_h$.
\end{proof}

\section{Preservation of Wildness Comparisons under Certain Operations}
In this section we shall show the "at least as wild as" property is preserved under suspension, rounded product 
and spin operations but is not preserved under the inflation operation.

\begin{theorem} 
If $(C,D,h) \in \mathscr{W}_n$, then so is $\big(\Sigma(C),\Sigma(D),\Sigma(h)\big)$, where $\Sigma$ denotes the 
suspension operator.
\end{theorem}
\noindent The proof is elementary: suspend an associated map $f_h$.

The following example shows that the inflation operator does not preserve the "at least as wild as" property.
\begin{example} 
If $C$ is at least as wild as  $D$  and  $\text{Infl}(C)$ is a crumpled cube
(equivalently for $n > 4$, $C$ has the Disjoint Disks Property), then 
$\mathrm{Infl}(C)$ might not be at 
least as wild as $\mathrm{Infl}(D)$. Suppose $D$  is a crumpled cube whose boundary is everywhere wild, and its inflation is also a crumpled cube. Then the only crumpled cube $C$ for which $\text{Infl}(C)$ is at least as wild as  
$\text{Infl}(D)$ is  $D$ itself. If  $H: \text{Bd } \text{Infl}(C)\rightarrow \text{Bd } \text{Infl}(D)$ is a 
homeomorphism, $H$ must send the wild set of $\text{Infl}(C)$ to cover the wild set of $\text{Infl}(D)$; that is, 
$H(\text{Bd } C \times \{0\}) \supset \text{Bd }D \times \{0\}$. Since no proper subset of an $(n-1)$-sphere
can cover another $(n-1)$-sphere, it follows that 
$H(\text{Bd } C \times \{0\}) = \text{Bd }D \times \{0\}$. As a result,
$H$ must send either of the obvious copies of $C$ in 
the boundary of the first inflation onto a copy of $D$ in the second.
\end{example}

\begin{definition} 
Given a crumpled $n$-cube $C$, we define its \textit{rounded product} with $I$,
denoted $\text{Round}(C \times I)$, as the crumpled $(n+1)$-cube in $\mathbb{R}^{n+1}$ bounded by
$\lambda\big(\text{Bd }(C \times I)\big)$ where $\lambda$ is an embedding that agrees with inclusion
on $(\text{Bd } C) \times [\frac{1}{3},\frac{2}{3}]$, is locally flat elsewhere, and where the image of $\lambda$ 
misses $\text{Int }C \times [\frac{1}{3},\frac{2}{3}]$.  Equivalently, $\text{Round}(C \times I)$ is obtained by 
attaching $(n+1)$-cells $B_+$ and $B_-$ to $C \times [\frac{1}{3},\frac{2}{3}]$ along $C\times \{\frac{2}{3}\}$ and $C \times \{\frac{1}{3}\}$, respectively, with both $\text{Bd }B_+ - (\text{Int }C \times \{\frac{2}{3}\})$ and $\text{Bd }B_- - (\text{Int }C \times \{\frac{1}{3}\})$ required to be $n$-cells.
\end{definition}

\begin{theorem} 
If $(C,D,h) \in \mathscr{W}_n$, then  $$\big(\mathrm{Round}(C \times I),\mathrm{Round}(D \times I),\mathrm{Round}(h \times \mathbbm{1})\big) \in \mathscr{W}_{n+1},$$ where $\mathrm{Round}(h \times \mathbbm{1})$ denotes any homeomorphisms between the
boundaries that extend $h \times \mathbbm{1}:\mathrm{Bd}$ $C \times [\frac{1}{3},\frac{2}{3}] \to \mathrm{Bd}$ $D 
\times [\frac{1}{3},\frac{2}{3}]$.
\end{theorem}
\begin{proof}
Since $(C,D,h) \in \mathscr{W}_n$, we have a typical associated map $f_h:C \to D$ extending $h$. Define 
$$F:\mathrm{Round}(C \times I) \to \mathrm{Round}(D \times I)$$ 
as $f_h \times \mathbbm{1}: C \times [\frac{1}{3},\frac{2}{3}] \to D \times [\frac{1}{3},\frac{2}{3}]$; extend to the $(n+1)$-cells 
$B_+$ and $\beta_+$ attached
along $C \times \{\frac{2}{3}\}$ and $D \times \{\frac{2}{3}\}$, respectively, so that no point of $\text{Int }B_+$ is sent to a boundary point of $B_+$, and do the same for the $(n+1)$-cells attached at the $\frac{1}{3}$-levels.	
\end{proof}

We introduce a method for spinning a crumpled $n$-cube $C$ that sometimes,
but not always, produces a crumpled $(n+k)$-cube.  It is closely related to the method of spinning a decomposition described in Section 28 of \cite{Daverman 1},
and we will make use of results from that section. The procedure depends
on a choice of an $(n-1)$-cell $\beta$ in $\text{Bd }C$.  For simplicity we will
tolerate using only those cells $\beta$ that are standardly embedded in $\text{Bd }
C$.  For $k>0$ the $k$-\textit{spin} of $C$ relative to $\beta$ is the decomposition
space $\text{Sp}^k(C,\beta) = C \times S^k /\mathscr{G}_\beta$, where $\mathscr{G}_\beta$
is the decomposition whose nondegenerate elements are $\{c \times S^k | c
\in \beta\}$.  This is a generalized $(n+k)$-manifold with boundary, and its
boundary is the image of $(\text{Bd }C - \text{Int }\beta) \times S^k$, the $k$-spin of the
$(n-1)$-cell $\text{Bd }C - \text{Int }\beta$, which is an $(n+k-1)$-sphere.  As a result,
$\text{Sp}^k(C,\beta)$ is a crumpled $(n+k)$-cube if and only if it embeds in
$S^{n+k}$.

Given crumpled $n$-cubes $C$ and $D$ plus a homeomorphism $h$ of $\text{Bd }C$ to $\text{Bd }D$, we
have a naturally defined homeomorphism $\text{Sp}^k(h)$ between the boundaries of certain $k$-spins.  To spell this out, let $q_C:(\text{Bd }C - \text{Int }\beta)  \times S^k
\to \text{Bd }\text{Sp}^k(C,\beta)$ and $q_D:\big(\text{Bd }D - \text{Int }h(\beta)\big) \times S^k 
\to \text{Bd }\text{Sp}^k(D,h\big(\beta)\big)$ denote the decomposition maps, appropriately
restricted. Define 
$$\text{Sp}^k(h): \text{Bd }\text{Sp}^k(C,\beta) \to \text{Bd }\text{Sp}^k(D,h(\beta))$$
as $q_D (h \times \mathbbm{1}) (q_C)^{-1}$, where
$h \times \mathbbm{1}: (\text{Bd }C - \text{Int }\beta)  \times S^k \to \big(\text{Bd }D - \text{Int }h(\beta)\big)
\times S^k$.

There is another effective way of studying $\text{Sp}^k(C,\beta)$.  Attach an exterior collar
$\lambda(\text{Bd }C \times [0,1])$ to $C$, with $\lambda(c,0) = c$ for all $c \in \text{Bd }C$.
 The union of $C$ and the collar is an $n$-cell $B^n$.  Let $G$ be the decomposition
of $B^n$ consisting of points and the arcs $\lambda(c \times [0,1])$, $c \in
\beta$. (The admissibility is satisfied.)   The $k$-spin of $B^n$ is topologically $S^{n+k}$
and $G$ gives rise to a cell-like decomposition $\text{Sp}^k(G)$.  The decomposition
space $\text{Sp}^k(B^n)/\text{Sp}^k(G)$ contains $\text{Sp}^k(C,\beta)$.  Hence, if $\text{Sp}^k(B^n)/\text{Sp}^k(G)$
is the $(n+k)$-sphere, then $\text{Sp}^k(C,\beta)$, being bounded by an
$(n+k-1)$-sphere there, must be a crumpled $(n+k)$-cube.  According to \cite[Theorem 28.9]{Daverman 1}, when $n+k \geq 5$,  
$\text{Sp}^k(B^n)/\text{Sp}^k(G)$ is topologically $S^{n+k}$
if and only if every pair of maps $\mu_1,\mu_2:I^2 \to B^n/G$ can be
approximated, arbitrarily closely,  by maps $\mu'_1,\mu'_2:I^2 \to B^n/G$ such that
\begin{equation}
\mu'_1(I^2) \cap \mu'_2(I^2) \cap \pi_G(\partial B^n) = \emptyset,
\tag{*}
\end{equation}
where $\pi_G$ denotes the decomposition map $\pi_G:B^n \to B^n/G$.  In our case,
(*) can be replaced with
\begin{equation}
\mu'_1(I^2) \cap \mu'_2(I^2) \cap \pi_G(\beta) = \emptyset,
\tag{**}
\end{equation}
since all nondegenerate elements of $G$ meet $\beta$, which means that both
singular disks $\mu'_1(I^2)$ and $\mu'_2(I^2)$ can be adjusted to avoid
$\pi_G(\partial B^n) - \pi_G(\beta)$.

In summary, $\text{Sp}^k(C,\beta)$ is a crumpled $(n+k)$-cube if and only if it
satisfies the Disjoint Disks Property at $\beta$, a property defined by
Condition (**).  Of course, the usual Disjoint Disks Property is more than
enough to assure that (**) holds.

\begin{lemma}  
Suppose $(C,D,h) \in \mathscr{W}_n$ and $\beta$ is an $(n-1)$-cell
standardly embedded in $\mathrm{Bd}$ $C$  such that $\mathrm{Sp}^k(C,\beta)$ is a crumpled
$(n+k)$-cube.  Then $\mathrm{Sp}^k\big(D,h(\beta)\big)$ is also a crumpled $(n+k)$-cube.
\end{lemma}
\begin{proof}
Since $\text{Sp}^k(C,\beta)$ is a crumpled cube, $C$ contains homotopy taming
sets $T_1,T_2$ such that
$T_1 \cap T_2 \cap \beta = \emptyset$.
Then $h(T_1), h(T_2)$ are homotopy taming sets for $D$ and
$h(T_1) \cap h(T_2) \cap h(\beta) = h(T_1 \cap T_2 \cap \beta) = \emptyset$,
which assures that $\text{Sp}^k\big(D,h(\beta)\big)$ is a crumpled $(n+k)$-cube.
\end{proof}

\begin{theorem} 
Suppose $(C,D,h) \in \mathscr{W}_n$ and $\beta$ is an $(n-1)$-cell
standardly embedded in $\text{Bd }C$  such that $\mathrm{Sp}^{n+k}(C,\beta)$ is a crumpled
$(n+k)$-cube.  Then $\big(\mathrm{Sp}^k(C,\beta),\mathrm{Sp}^k\big(D,h(\beta)\big),\mathrm{Sp}^k(h)\big)\in \mathscr{W}_{n+k}$.
\end{theorem}
\begin{proof}
Let $p_C:C \times S^k  \to \text{Sp}^k(C,\beta)$ and $p_D:D
\times S^k \to  \text{Sp}^k\big(D,h(\beta)\big)$ denote the decomposition maps. 
Let $f_h:C \to D$ be a map associated with $(C,D,h)$.  Define $F: \text{Sp}^k(C,\beta)
\to \text{Sp}^k(D,h(\beta))$ as $p_D (f_h \times \mathbbm{1})(p_C)^{-1}$,
where $f_h \times \mathbbm{1}: C \times S^k \to D \times S^k$.
\end{proof}


\section{Strict Wildness Considerations}
As mentioned in Section 3, the definition of "at least as wild as" does not provide a partial order on the collection of crumpled $n$-cubes. To show that the relation fails to be antisymmetric, we present a pair of crumpled $n$-cubes with non-homeomorphic wild sets --- therefore assuring the two are topologically distinct --- yet where each is at least as wild as the other. 

For greater clarity, we shall introduce a standard flattening technique to construct such an example.
\begin{definition} 
Given a crumpled $n$-cube $C$  and a compact set $X\subset \text{Bd }C$, the crumpled $n$-cube $C_X$ is a 
\textit{standard flattening} relative to $X$ if there exists a homeomorphism $h_X$ of $\text{Bd }C$ to 
$\text{Bd }C_X$ such that $(C,C_X,h_X)\in \mathscr{W}_n$, 
$h_X|X = \mathbbm{1}|X$, Bd $C_X - X$ is locally collared 
in $C_X$, $C \subset C_X,$  and $C_X$ admits a 
strong deformation retraction 
$r:C_X \to C$ such that $rh_X = \mathbbm{1}|\mathrm{Bd }C$
and all nontrivial point preimages under $r$ are sent to 
$\mathrm{Bd }C - X$. 
\end{definition}

Over the course of the next three paragraphs we show the existence of 
a standard flattening for  every crumpled cube $C$ and closed subset $X$ of $\text{Bd }C$.
Treat $C$ as a closed $n$-cell- 
complement. Name an embedding $\lambda$ of $\text{Bd }C \times I$  giving a collar Bd $C$ in that 
$n$-cell $S^n-\text{Int }C$, with $\lambda(s,0)=s$ for all $s\in \text{Bd }C$. Find a continuous function 
$\mu:\text{Bd }C\rightarrow [0,1]$ such that $X=\mu^{-1}(0)$. Define $C_X$ as 
$$C\cup \{\lambda(s\times [0,\mu(s)]), \text{ where }s\in \text{Bd }C\}.$$

The retraction $r:C_X \to C$ each 
$\lambda(s\times [0,\mu(s)])$ to $\lambda(s\times 0)$;
it should be obvious why $r$ is a strong deformation 
retraction.  For later convenience, further restrict 
$\mu$ so that $\text{diam } \lambda(s' \times [0,\mu(s')]) < \text{dist}(s',X)$ for all $s' \in \text{Bd } C - X$.

Next, we will apply the (proof of the) Homotopy Extension Theorem to the pair $(C - X, \text{Bd }C - X)$ and the ANR $C_X - X$.
Consider the inclusion of $C - X$  in the target $C_X - X$ and the homotopy of $\text{Bd }C - X$ beginning with 
inclusion and ending with the homeomorphism $h_X:\text{Bd }C - X \rightarrow \text{Bd }C_X -X$ sending
$s = \lambda(s,0)$ to $\lambda\big(s,\mu(s)\big)$.
The track of this homotopy at $s$ has diameter less than 
$\text{dist}(s,X)$. There is a neighborhood $N$  of $\big((C - X) \times \{0\}\big) \cup \big((\text{Bd }C - X) \times 
I\big)$ over which our partial homotopy extends (into  $C_X - X$).  Name this extension
as  $\psi$.  Each  $s \in  \text{Bd }C - X$  has a neighborhood $O_s$ there such that $O_s \times I \subset N$ and
$\text{diam }\psi(s \times I)$ has diameter which is less than $\text{dist}(s,X)$. Let  $O$ be the union of all these 
$O_s$. Find a Urysohn function $u:C_X \rightarrow [0,1]$ with $u(\text{Bd }C-X) = \{1\}$ and
$u(C - O) = \{0\}$.  Then define  $\Psi:(C - X)  \times I \rightarrow C_X - X$ as $\Psi(s,t)=\psi(s,t \cdot u(s))$.
The claim is that  $\Psi$  extends via projection to  $X$  on $X \times I$ to give a map  $C \times I \rightarrow 
C_X$.  This function $\Psi$ is continuous at $X$: for points $y$ within $\epsilon$ of $X$, either
$\Psi(y \times I)=y$ or $\text{diam }\Psi(y \times I) < \text{dist}(y,X) < \epsilon$, so 
$\text{dist}\big(\Psi(y,t),x\big) < 2 \epsilon$, assuring continuity.

Thus, the function $h_X$ extends via the identity on $X$
to a homeomorphism $h_X:\mathrm{Bd }C \to \mathrm{Bd }C_X$.
Note that $rh_X = \mathbbm{1}|\mathrm{Bd }C$.  
The desired map  $f:C \rightarrow C_X$ is almost $f(c) = \Psi(c,1)$, where $c\in C$.
A map like this does extend the homeomorphism  $h_X$  between the boundaries. The only problem is that $f$  could send some 
point of  $\text{Int }C$  to a point of $\text{Bd }C_X - X$.  This can be fixed as we
did when improving a map  $C\rightarrow B^n$ to assure no point of $\text{Int }C$ gets sent to $\text{Bd }B^n$.
Note that with any standard flattening $C_X$ of $C$, one can regard  $X$ as a
subset of $\text{Bd }C_X$.

\begin{remark}
A standard flattening $(C,C_X,h_X)$ can easily result in a
relatively uninteresting example in which $C_X$ turns out to be an $n$-cell.  That happens
whenever $C$ has a homotopy taming set $T$ that misses $X$.  Consequently, when
$C$ is the sort of crumpled cube for which any countable dense subset $J$ of
$\text{Bd }C$ is a homotopy taming set, then $\dim X \leq n-2$ implies $C_X$ is an $n$-cell.
However, if $X$ is the closure of some open subset of $C$'s wild set, then $C_X$
is truly wild at each point of $X$.
\end{remark}

\begin{lemma} 
Suppose $C$ is a crumpled $n$-cube, $\{X,Y\}$
a pair of compact sets with $Y \subset X \subset \mathrm{Bd}$ $C$,
$(C,C_X,h_X)$ and $(C,C_Y,h_Y)$ are standard flattenings
of $C$ with respect to $X$ and $Y$, respectively, and
$\big(C_X,(C_{X})_Y,(h_{X})_Y\big)$ is the standard flattening of $C_X$ with
respect to $Y$.  Then it can be arranged that $(C_{X})_Y = C_Y$
and $(h_{X})_Y h_X = h_Y$.
\end{lemma}
\begin{proof}
The map $\mu:\text{Bd } C \to [0,1]$ producing the standard flattening  $C_X$
should be chosen so $\mu(\text{Bd }C) \subset [0,1)$.  Then there exists a
collar $\lambda_X(\text{Bd } C_X \times [0,1])$ on $\text{Bd }C_X$  in $C_X$ such that
$$\lambda_X(h_X(s) \times [0,1]) \subset \lambda(s \times [\mu(s),1]),$$
whenever $s\in \text{Bd }C$. The conclusion follows.
\end{proof}

\begin{theorem} 
Let $C$ and $D$ be crumpled $n$-cubes with wild sets $W_C, W_D$, 
respectively, and let $h:\mathrm{Bd}$ $C \to \mathrm{Bd}$ $D$ be a homeomorphism such
that $W_D \subset h(W_C)$.  Set $X = h^{-1}(W_D)$.  Then 
$(C,D,h) \in \mathscr{W}_n$ if and only if
$(C_X,D,hh^{-1}_X) \in \mathscr{W}_n$.
\end{theorem}
\begin{proof}

The reverse implication follows from Definition 5.1,
Lemma 5.1 and the transitivity property for appropriate 
pairs of triples belonging to $\mathscr{W}_n$.

Turning to the forward implication, we consider
$(C,D,h) \in \mathscr{W}_n$ and name a map
$f_h:C \to D$ associated with $h$.  There is a retraction
$r:C_X \to C \subset C_X$  that restricts to $h^{-1}$ on 
Bd $C_X$ and otherwise satisfies  
$r(C_X - C) \subset \mathrm{Bd }C - X$.
Then $f_hr:C_X \to D$
restricts to $hh^{-1}_X $ on Bd $C_X$.
The only problem with $f_hr:C_X \to D$ is that it sends some points of
Int $C_X - \mathrm{Int }C$ to Bd $D$.  However, the images of those troublesome points miss $W_D$, so modifications of a now familiar sort
can be used to push those images off Bd $D$, as required.
\end{proof}

\begin{example} 
A pair of crumpled $n$-cubes with non-homeomorphic wild sets such that each is at least as wild as the other.

In $\mathbb{R}^{n-1}$, identify a countable collection of round $(n-2)$-spheres, no two of which intersect, plus a point to which these spheres 
converge. Let $Z$ be the union, and let $Z^*$ be $Z$  with one of the spheres removed.  One still can show that there 
exists a homeomorphism  $h$ of  $\mathbb{R}^{n-1}$ to itself such that  $h(Z) = Z^*$.

Next, label the $(n-1)$-balls bounded by the spheres in  $Z$ as 
$$B_1,B_2,B_3,\dots,B_i,\dots,$$
with the understanding that $B_1$  misses $Z^*$, and that homeomorphism $h$ of $\mathbb{R}^{n-1}$ to itself
sends $B_i$ to $B_{i+1}$.  Extend $h$ to a homeomorphism  $H$ of  $\mathbb{R}^{n-1} \times [0,\infty)$  to itself  
that takes each  $p\times [0,\infty)$ to  $h(p)\times [0,\infty)$ and carries  $B_i \times [0,\frac{1}{i}]$ onto  
$B_{i+1} \times [0,\frac{1}{i+1}]$. 

Find a crumpled cube $C$ whose boundary is locally flat modulo a
simple closed curve $J$ standardly embedded in Bd $C$.  Replace each $n$-cell $B_i 
\times [0,\frac{1}{i}]$ with a copy of  $C \subset B_i \times [0,\frac{1}{i}]$, i.e. embedding a copy of $C$  in $B_i 
\times [0,\frac{1}{i}]$ so that the image of $C$ contains all of this $n$-cell's boundary not in
$\mathbb{R}^{n-1} \times \{0\}$. For later reference, denote $\big(\text{Bd }B_i \times [0,\frac{1}{i}]\big) \cup (B_i 
\times \{\frac{1}{i}\})$ as $\beta_i$.  Make sure the image of $J$ misses $\beta_i$.
Once we do this replacement for the first  $B_1 \times [0,1]$, do it in a way that is compatible with  $H$ in the 
remainder. That is, do it so the homeomorphism  $H$  restricted to  $\beta_i$ extends to a homeomorphism from the copy 
of $C$ in  $B_i \times [0,\frac{1}{i}]$ to the copy of $C$ in  $B_{i+1} \times [0,\frac{1}{i+1}]$.

Let $K$ be the subset of  $\mathbb{R}^{n-1} \times [0,\infty)$ obtained by deleting the cells $B_i\times 
[0,\frac{1}{i}]$ from $Z$ and replacing with the copies of $C$, and let $K^*$  be the space obtained when we leave 
$B_1\times I$ as it is and replace all the others.  $H$  extends to give a homeomorphism of $K$ onto $K^*$,
The one point compactifications of $K$  and  $K^*$  are crumpled cubes, and  $H$ extends to give a homeomorphism $H^*$ 
of those crumpled cubes. Let's use $\overline{K}$ and $\overline{K^*}$  as the names for
these crumpled cubes, the compactifications.

Let  $W$  denote all of the wild set of  $\overline{K}$  except the interior of an open subarc $A$ of the copy of $J$ 
in the first replacement $C$. So $W$ consists of a point, an arc, and a sequence of simple closed curves whose 
diameters go to 0. The standard flattening gives  $(\overline{K},\overline{K}_W,h_W) \in \mathscr{W}_n$. Let $W'$ denote 
all of the wild set of $\overline{K}$  that misses
$B_1 \times I$.  Looking at $J$ in that first replacement in another way, $W'$ is the portion of the wild set of 
$\overline{K}_W$  
except for the component that is an arc.  We have a standard flattening of $\overline{K}_W$ relative to $W'$,
which means $\big(\overline{K}_W,(\overline{K}_W)_{W'},h_{W'}\big) \in \mathscr{W}_n$. Here $\overline{K^*}$  can be regarded 
as a standard flattening of $\overline{K}$
relative to all of the wild set outside that copy of $J$ in the first replacement crumpled cube.
Standard flattenings with respect to the same subset are homeomorphic, so $(\overline{K}_W)_{W'} $, which  is flattening first 
relative to $W$ and then relative to $W' \subset  W$,   is the same as flattening relative to $W'$.   As a result, 
$(\overline{K}_W)_{W'}$ is homeomorphic to $\overline{K^*}$ and we know  $\overline{K^*}$ is homeomorphic to 
$\overline{K}$.  This certifies that $\overline{K}$ is at least as wild as $\overline{K}_W$
and $\overline{K}_W$ is at least as wild as $(\overline{K}_W)_{W'} \cong \overline{K}$.
\end{example}

\begin{definition} 
$C$ is \textit{strictly wilder than} $D$ if and only there exists a homeomorphism $h$ such that $(C,D,h)\in \mathscr{W}_n$ but there is no homeomorphism $H$ such that $(D,C,H)\in \mathscr{W}_n$.
\end{definition}
Definition 5.2 imposes a strict partial order on the collection of all
crumpled $n$-cubes (up to homeomorphism).

\begin{theorem} 
Suppose $D$ is a closed $n$-cell-complement whose wild set $W$ is a
proper subset of $\mathrm{Bd}$ $D$.  Then there exists a closed $n$-cell-complement $C$ strictly
wilder than $D$ such that $\mathrm{Bd}$ $C$ is wild at each of its points.
\end{theorem}
\begin{proof}
Triangulate $\text{Bd }D - W$ and list the $(n-1)$-simplexes
$$\Delta_1, \Delta_2,\dots,\Delta_k,\dots$$
of this triangulation.  (If $W$ is topologically a polyhedron tamely
embedded in $\text{Bd }D$, this can be a finite list and the $\Delta_i$ can be allowed to touch $W$ in
their boundaries; otherwise, however, require the diameters of the
$\Delta_k \to 0$  as $k \to \infty$.)

Since $\text{Bd }D$ is locally flat at all points of $\text{Int }\Delta_i$, $\Delta_i$
can be thickened to an $n$-cell $B_i$ in $D$ such that $\Delta_i$ is a standardly 
embedded subset of $\text{Bd }B_i$ and $B_i \cap B_j \subset \Delta_i \cap \Delta_j$ for all $i \neq j$.

Let $K$ be a crumpled $n$-cube whose boundary is locally flat modulo an
$(n-1)$-cell $A$ standardly embedded in $\text{Bd }K$. Also require that $\text{Bd }A$ be tame in space
(i.e., some homotopy taming set for $K$ misses $\text{Bd }A$).  (To get such a $K$,
modify the Bing's construction of \cite{Bing 1} to
generate a crumpled 3-cube wild at the points of a 2-cell with tame
boundary, and suspend as often as needed.) Then addition to $K$ of a tapered
(exterior) collar on $\text{Int }A$ produces an $n$-cell $B^n$ containing $A$ with $\text{Cl }(B^n -K)
- \text{Bd }A$ equal to that tapered collar. Equate each $B_i$ with a copy of $B^n$ so
as to identify a copy  $K_i$ of $K$ in each $B_i$; do this so
$K_i \cap \text{Cl }(D - B_i)$ corresponds to $\text{Cl }(\text{Bd }K - A)$.  It follows that $(D -
\cup_i B_i) \cup K_i$ is a crumpled $n$-cube $C$; in other words, $C$ results
from $D$ by deleting all the tapered collars and taking the closure of what
remains.  It should be immediately obvious that $\text{Bd }C$ is everywhere wild
and that $\text{Bd }C$ contains $W$.

Build an exterior collar on $C$ by first appending the tapered collars in
$B_i$ on the various $K_i$.  The union equals $D$.  When it is combined with an
exterior collar on $\text{Bd }D$, we have a collar on $C$.  A standard flattening $C_W$
of $C$ then equals $D$; thus, we have ($C,D =C_W,h_W) \in \mathscr{W}_n$.

Since $\text{Bd }C$ is everywhere wild and $\text{Bd }D$ is not, $C$ is strictly wilder than $D$.
\end{proof}

\begin{theorem} 
For any crumpled $n$-cube $D$, $n>3$, there exists another crumpled $n$-cube $C$ and homeomorphism $h$ 
with $(C,D,h) \in \mathscr{W}_n$ such that every associated map $f_h:C \to D$ extending $h$ restricts to an epimorphism 
$(f_h)_\#:\pi_1(\mathrm{Int}$ $C) \to \pi_1(\mathrm{Int}$ $D)$ having non-trivial kernel.
\end{theorem}
\begin{proof}
Given any crumpled cube $D$ we can find a Cantor set $X$ which misses some homotopy taming set for $D$ (i.e. $X$ is tame
in space when  $D$  is embedded in $S^n$ as a closed $n$-cell-complement). We claim that there exists $(C,D,h) \in \mathscr{W}_n$ such that, for every homotopy taming set $T$ for $C$,
$T \cap h^{-1}(X)$ is nonempty.  In other words, $h^{-1}(X)$ fails to be tame in space. The key is to produce an 
$(n-1)$-sphere $S$ in  $X \cup \text{Int }D$ that is locally flat modulo $X$ and which 
contains $X$ as a standard Cantor set in $S$.  By \cite{Kirby 1} $S$ is standardly embedded in $S^n$, so it bounds two $n$-cells, $B$ and $B'$, with $B \subset D$.  
Remove $B$ from $D$ and replace it with a crumpled cube $K$ locally flat modulo a Cantor set $Z$ wild in space but tame in $\text{Bd }K$; specifically, attach  $K$ to $D - \text{Int }B$ via a homeomorphism $\theta:\text{Bd }B \to \text{Bd }K$ such that $\theta(X) = Z$.  Let $C$ be the result of the replacement. Keep in mind that  $K$  can be put into $S^n$ 
as a closed $n$-cell-complement, so $S^n = B' \cup K \supset C$; in short,  
$C$  is a crumpled cube.

Note that $S - X$ is simply connected.  Hence, 
$$\pi_1(\text{Int }C) \cong \pi_1(\text{Int }C- K)\ast \pi_1(\text{Int }K).$$ 
Consider any loop $\gamma:\partial I^2 \to \text{Int }C$ with
image in $\text{Int }K$ and any map  $f_h:C \to D$ associated with the obvious
homeomorphism between $\text{Bd }C$ and $\text{Bd }D$.  Then $f_h \gamma$ extends
to a map $\Gamma:I^2 \to D$; since $T$ is a homotopy taming set for $D$, $\Gamma$ can
be approximated by a map $\Gamma'$ agreeing with $\Gamma$ on $\text{Bd }I^2$, with $\Gamma'(I^2)
\subset T \cup \text{Int }D$.  Set $U = D - f_h(K)$ and let $V$ denote the component
of $U - \text{Bd }D$ whose closure contains $\text{Bd }D - X$.  Find a disk with holes $P$ in
$I^2$ with $\partial I^2 \subset P$, $\Gamma(P) \subset \text{Int }D$, and
$\partial P - \partial I^2 \subset V$. This means that the subgroup of
$\pi_1(\text{Int } D)$ carried by $f_h \gamma(\partial I^2)$ is in the normal closure of
$\pi_1(V)$, and elements $\alpha_1, \dots, \alpha_m$ determined by the components of
$\partial P - \partial I^2$. Let $V'$ denote the component of
$f_h^{-1}(V)$ whose closure contains  $\text{Bd }C - Z$.  By Lemma 3.1, the
restriction of $f_h$ induces an epimorphism $\pi_1(V') \to \pi_1(V)$, so 
$\pi_1(V')$ contains elements $\alpha'_i$ sent by $f_h$ to $\alpha_i$ $(i=1,\dots,m)$. Observe that $V' \subset 
\text{Int }C - K$.  If $f_h$ also restricted to give an
isomorphism $\pi_1(\text{Int }C) \to \pi_1(\text{Int }D)$, the subgroup of $\pi_1(\text{Int }C)$
carried by $\gamma(\partial I^2)$ would be in the normal closure of the
$\alpha_i$, and hence on the normal closure of $\pi_1(\text{Int }C - K)$ with respect
to $\pi_1(\text{Int }C)$, an impossibility.
\end{proof}

Given one crumpled  cube $D$, Theorem 5.2 presents a method for constructing another crumpled cube $C$  at least as wild as $D$; in most circumstances 
$C$ will be strictly wilder than $D$ by virtue of having a larger wild set, topologically distinct from that of  $D$.  Theorem 5.3 accomplishes a similar purpose without changing the topological type of the wild set but instead increasing the wildness of the boundary sphere. 
Decomposition theory affords a very general technique for creating additional examples.

\begin{example}
A decomposition theory technique for producing a crumpled cube at least as wild a given crumpled cube $D$  and having topologically equivalent wild set.
Given $D$, locate an $n$-cell $B$ in $D$ with $B \cap \text{Bd }D = X$, and consider any cell-like decomposition  $G$ of $S^n$ whose nondegenerate elements are subsets of $B$,
each of which meets $\text{Bd }B$ in a single point of $X$.  Let $N_G$ denote the union of those nondegenerate elements. To be truly effective,
assume there exists at least one loop $\gamma$ in $\text{Int }B - N_G$
which is homotopically essential in $S^n - N_G$.  Let $\pi_G: S^n \to S^n/G$
denote the decomposition map.  Since $\pi_G(N_G) \subset \pi_G(\text{Bd }B)$ and $\pi_G|\text{Bd }B$ is 1-1, $S^n/G$ is finite dimensional.  Often $G$ will be shrinkable and $S^n/G$ will be homeomorphic to $S^n$;
this can be assured by imposing additional restrictions on $G$. If 
shrinkability fails, one can suspend all the relevant data and obtain examples in higher dimensions.  We continue by describing how to proceed when $G$ is shrinkable.  Set $C = \pi_G(D)$. It is a crumpled cube since $\pi_G|\text{Bd }C$ is 1-1; moreover, the function $(\pi_G)^{-1}$ restricts to a homeomorphism of $\text{Bd }C$ onto $\text{Bd }D$.  To see that $(C,D,h) \in \mathscr{W}_n$, note that $(\pi_G)^{-1}|\pi_G(C - \text{Int }B)$
extends to a map $f_h:C \to D$ such that $f_h\big(\pi_G(\text{Int }B)\big) \subset \text{Int }B$, just as in the proof of Theorem 3.1. The special loop $\gamma$ has image under $\pi_G$ that cannot be shrunk in $\text{Int }C$, but $f_h \pi_G(\gamma) \subset \text{Int }B$ must be contractible in $\text{Int }D$.  
\end{example}

Theorem 5.3 suggests there is no maximal element in the 
"strictly wilder than" partial ordering.  We have not established that fact.
However, we do have the following corollaries.

\begin{corollary} 
Let  $C$ denote a crumpled $n$-cube such that $\pi_1(\mathrm{Int}$ $C)$ is
finitely generated.  Then $C$ is not a maximal element in the "strictly
wilder than" partial ordering.
\end{corollary}
\begin{proof}
By Theorem 5.3, there exist a crumpled $n$-cube $\tilde{C}$ and a
homeomorphism $h$ such that $(\tilde{C},C,h) \in \mathscr{W}_n$ and $\pi_1(\text{Int }
\tilde{C})$ is a non-trivial free product $G*\pi_1(\text{Int }C)$.   Hence, by Grushko's Theorem the number of generators of $\pi_1(\text{Int }\tilde{C})$ must be greater
than the rank of $\pi_1(\text{Int }C)$, so $\tilde{C}$ must be strictly wilder than $C$.
\end{proof}

\begin{corollary} 
Let  $C$ denote a crumpled $n$-cube such that $\pi_1(\mathrm{Int}$ $C)$ is
a simple group.  Then $C$ is not a maximal element in the "strictly
wilder than" partial ordering.
\end{corollary}
\begin{proof}
Again the construction of Theorem 5.3 provides a crumpled $n$-cube $\tilde{C}$ at least as wild as $C$, the fundamental group of which is a non-trivial free product.  Here $C$  cannot be at least as wild as 
$\tilde{C}$ --- there can be no epimorphism of $\pi_1(\text{Int }C)$  to the non-simple group $\pi_1(\text{Int }\tilde{C})$.
\end{proof}
\begin{corollary} 
If  $C$ is a crumpled n-cube such that $\pi_1(\mathrm{Int}$ $C)$
is a torsion group, then $C$ is not a maximal element in
the "strictly wilder than" partial ordering.
\end{corollary}

The reader can confirm the existence of infinite families totally ordered
under the "strictly wilder than" relation.

The $n$-cell $B^n$ is the unique minimal element in the partial ordering
on the collection of all closed-$n$-cell-complements.

\begin{theorem} 
Every non-trivial crumpled $n$-cube is strictly wilder than the $n$-cell $B^n$.
\end{theorem}
 \begin{proof}
That any  crumpled $n$-cube $C$ is at least as wild as $n$-cell $B$ has been established in Theorem 3.1. It suffices to show that 
$(B^n,C,H)$ is never in $\mathscr{W}_n$, no matter which homeomorphism $H:\text{Bd } B^n \rightarrow \text{Bd }C$ is 
under consideration. The empty set is a homotopy taming set for $B^n$.
If $(B^n,C,H) \in \mathscr{W}_n$, then by Theorem 3.5, $\emptyset$  is also a homotopy taming set for $C$.
In other words,  $\text{Int }C$  is 1-ULC.  That can only occur when  $C$  is an $n$-cell.
\end{proof}
The results immediately below are direct applications of Corollary 3.14 and Theorem 3.15.

\begin{theorem} 
If a closed  $n$-cell-complement $C$ is locally flat modulo a wild set, a closed $n$-cell complement $D$ is locally flat modulo
a tame set and $C$ is at least as wild as $D$, then $C$ is strictly wilder than $D$.
\end{theorem}

\begin{theorem} 
If a crumpled $3$-cube $C$ is at least as wild as another crumpled $3$-cube $D$ and $C$ has more non-piercing points than 
$D$, then $C$ is 
strictly wilder than $D$.
\end{theorem}

We conclude this section with an open question.
\begin{question}
Are there any maximal elements in the partial order constructed by "strictly wilder than"?
\end{question}

\section{Sewings of Crumpled Cubes}
The triple $(C,D,h)\in \mathscr{W}_n$ automatically gives rise to a sewing of the two  crumpled cubes by identifying each point 
$x$ of $\text{Bd }C$ with the point $h(x)$ on $\text{Bd }D$. The associated sewing space is denoted $C \cup_h D$. 
It can be viewed as a decomposition space arising  
from a cell-like decomposition of $S^n$ into points
and the fiber arcs of an $n$-dimensional annulus.
This section concentrates on the interplay between $(C,D,h)$ being in $\mathscr{W}_n$ and the sewings $h$  that yield $S^n$.

\begin{theorem}
Suppose $(C,D,h) \in \mathscr{W}_3$,  $C^*$ is
another crumpled $3$-cube and $\theta:\mathrm{Bd}$ $C \to \mathrm{Bd}$ $C^*$ is a homeomorphism such that $C \cup_\theta C^*=S^3$.  
Then $D \cup_{\theta h^{-1}} C^* = S^3$.
\end{theorem}
\begin{proof}
By Eaton's characterization of the sewings of crumpled 3-cubes that yield $S^3$
\cite{Eaton},  $C,C^*$ have homotopy taming sets $T,T^*$, respectively, such that 
$\theta(T) \cap T^* = \emptyset$.  Then $h(T)$ is a homotopy taming set for $D$,
and clearly $\theta h^{-1}(h(T)) \cap T^* = \emptyset.$  Hence, again by 
\cite{Eaton}, $D \cup_{\theta h^{-1}} C^* = S^3$.
\end{proof}

The same argument fails in higher dimensional
settings  since the homotopy taming set mismatch feature is a sufficient but not 
a necessary condition for a sewing of crumpled cubes 
to yield $S^n$.

We will make use of the following Controlled Homotopy Extension Theorem.

\begin{theorem}
For each $\epsilon > 0$
there exists $\delta > 0$ such that
given any map $f:X \to Z$ of a normal space to $Z$
and any map $F_A:A \to Z$ defined on a closed subset $A$
of $X$  which is $\delta$-close to $f|A$,
then $F_A$ admits a continuous extension $F:X \to Z$ which is $\epsilon$-close to $f$.
\end{theorem}
\begin{proof}
Choose a $\delta > 0$ for which any two $\delta$-close 
maps are $\epsilon$-homotopic and then
reuse the motion control aspect of the proof that standard flattenings exist.
\end{proof}

\begin{theorem}
Suppose $(C,D,h) \in \mathscr{W}_n$, $n > 4$, $C^*$ is
another crumpled $n$-cube and $\theta:\mathrm{Bd}$ $C \to \mathrm{Bd}$ $C^*$ is a homeomorphism such that $C \cup_\theta C^*=S^n$.  
Then $D \cup_{\theta h^{-1}} C^* = S^n$.
\end{theorem}
\begin{proof}
Since $\Sigma = D \cup_{\theta h^{-1}} C^*$ is the cell-like image of $S^n$, by Edwards's Cell-Like Approximation 
Theorem \cite{Edwards} (or \cite{Daverman 1})
it suffices to prove that this sewing space satisfies the Disjoint Disks Property.  
To start the process, consider maps  $\psi_1,\psi_2:I^2 \to \Sigma$.
We will produce approximating maps with disjoint images  in six steps;
the only non-routine step, where the hypothesis that $C \cup_{\theta} C^* = S^n$ comes into play, is the final one. 

Step 1: Approximating to make the preimages $Z_1,Z_2$ of Bd $D$ = Bd $C^*$  1-dimensional.  Generically, for $e \in \{1,2\}$ $Z_e = \psi^{-1}_e(\mathrm{Bd }D)$ 
"ought to be" 1-dimensional.  If not, pushing certain points one at a time into
$\Sigma - \mathrm{Bd }D$, subject to convergence controls, 
we can readily modify $\psi_e$  slightly 
so that the new map (still called $\psi_e$) sends a countable,
dense subset of $I^2$ to $\Sigma - \mathrm{Bd }D$,
which assures 1-dimensionality. 

Step 2: Approximating to make images of  $Z_1,Z_2$ disjoint, 1-LCC subsets of $D$ and $C^*$.  Let $T,T^*$ denote $\sigma$-compact, homotopy taming sets 
for $D,C^*$, respectively.  According to \cite{Daverman 7}, $T$ and 
$T^*$ can be taken to have dimension at most 1.
Set $X_e = \psi^{-1}_e(D)$ and $Y_e = \psi^{-1}_e(C^*)$,
and note that $Z_e = X_e \cap Y_e$.
Approximate each $\psi_e|Z_e$ by an embedding $\lambda_e:Z_e \to \mathrm{Bd} D$ such that
$\lambda_1,\lambda_2$ have disjoint images in 
Bd $D - (T \cup h\theta^{-1}(T^*))$.
Require $\lambda_e$ to be so close to $\psi_e|Z_e$ that,
by the Controlled Homotopy Extension Theorem,
$\lambda_e$ extends to a map $\psi'_e:I^2 \to \Sigma$ close to $\psi_e$.
This must be done in two seperate operations, one extending $\lambda_e$
to an approximation of $\psi_e|X_e:X_e \to D$  and the other extending 
$\lambda_e$ to an approximation of $\psi_e|Y_e:Y_e \to C^*$.
It should be clear that $\lambda_e(Z_e) = \psi'_e(Z_e)$ is a 
1-LCC subset of each crumpled cube; for instance, since $T$ is a homotopy taming set for $D$, any map $I^2 \to D$ can  be approximated  
by a map $I^2 \to T \cup \mathrm{Int }D$, the image of which avoids 
$\lambda_e(Z_e)$.  It may be worth observing that,
unfortunately, $Z_e$ is not the complete preimage 
of Bd $D = \mathrm{Bd} C^*$ under $\psi'_e$. 

Step 3: Approximating to make the image of each $Z_e$ disjoint from the other
singular disk.  This step makes use of properties of homotopy taming sets.
The maps $\psi'_e|X_e:X_e \to D$ and
$\psi'_e|Y_e:Y_e \to C^*$ can be approximated,
fixing $\psi'_e|Z_e$.
by maps $\psi^*_e:X_e \cup Y_e \to \Sigma$ such that 
$$\psi^*(X_e - Z_e) \subset T \cup \mathrm{Int} D \mathrm{~ and ~} 
\psi^*(Y_e - Z_e) \subset T^* \cup \mathrm{Int} C^*.$$
It follows that
$$\psi^*_1(Z_1) \cap \psi^*_2(I^2) = \emptyset = \psi^*_1(I^2) \cap \psi^*_2(Z_2).$$
Although we will not write anything further to this effect explicitly, in successive steps we impose controls to maintain this disjointness feature.

Step 4: Approximating to make preimages of Bd $D$ in 
$I^2 - Z_e$ 0-dimensional.  This is a standard
operation.  Working in $D$ and $C^*$ separately, we
use the fact that the boundary is 0-LCC in the crumpled 
cube to approximate $\psi^*_e|X_e - Z_e$ 
and $\psi^*_e|Y_e - Z_e$ 
by maps $\Psi^*_1,\Psi^*_2:I^2 \to \Sigma$ 
that put the
1-skeleta of systems of finer and finer triangulations of the respective domains into the interiors of the relevant crumpled cubes.

Step 5: Covering the preimage of Bd $D$ in $I^2 - Z_e$
by a null sequence of pairwise disjoint disks.
Set $P_e = (\psi^*_e)^{-1}(\mathrm{Bd }D) - Z_e$,
and express it as a countable union of compact, 0-dimensional sets $K_1,K_2,\ldots$. Cover $K_1$ by a finite 
collection $\xi^1_e,\ldots,\xi^k_e$ of small disks (more about how small in the next step) in $I^2 - Z_e$ whose boundaries
miss $P_e$.  Then cover the part of $K_2$ not covered by the
$\xi^j_e$ by a finite collection of much smaller disks.
The latter should be pairwise disjoint and disjoint from the first collection. Continue in the same way.

Step 6: Approximating to make the intersection of singular disks disjoint. The heart of the matter is to make
the intersections of those disks disjoint from 
$\mathrm{Bd }D = \mathrm{Bd }C^*$.
Let $f_h:C \to D$ be a map associated with $(C,D,h)$ being in 
$\mathscr{W}_n$.  Extend  $f_h$ to a map  $F_h:C \cup_{\theta} C^* \to  \Sigma = D \cup_{\theta h^{-1}} C^*$ via the identity on $C^*$.  

The idea is to cover the preimages of Bd $D$ by null sequences of pairwise disjoint disks $\xi^j_e$ in $I^2 - Z_e$
as in Step 5, to approximately lift (with respect to $F_h$)
the restriction of $\Psi^*_e$ on those disks to maps into
$C \cup_{\theta} C^* = S^n$,
to adjust the lifted maps on these disks to pairwise disjoint 
embeddings in $S^n$ and to apply $F_h$.

Lifting singular disks mapped into $C^*$ is no problem, since
$F_h$ restricts to a homeomorphism over $C^*$.
In this paragraph we describe how to approximately lift certain singular disks mapped into $D$.
Given $\epsilon > 0$ choose $\delta > 0$ such that subsets $A$ of $C \cup_\theta C^*$ having diameter less than $\delta$  are mapped via $F_h$ to sets of diameter less than $\epsilon/2$.  Next, identify $\delta' > 0$
for which loops in $C \subset C \cup_\theta C^*$ of diameter less than  $\delta'$ bound singular disks in $C$  of diameter less than $\delta$. 
Build an open cover $U_1,\ldots,U_k$ of Bd $D$ in $D$ by connected open subsets of $D$ that meet Bd $D$ in connected sets and for which each $(F_h|C)^{-1}(U_i)$ has diameter less than $\delta'$.  Then find a compact neighborhood $Q$ of Bd $D$  in $D$  covered by the $U_i$, and let $\eta \in (0,\epsilon/4)$ be a Lebesgue number for this cover of $Q$. By Lemma 3.1 and this choice of $\eta$ each $\eta$-small loop $\gamma$ in 
$\mathrm{Int}D \cap Q$  is homotopic in the intersection of  Int $D$ with one of the $U_i$ 
to the image of a loop $\gamma'$ in Int $C$, and $\gamma'$ bounds a singular disk $\xi'$  in $C$ whose image under $F_h$ has diameter less than $\epsilon/2$.  We construct null sequences
$\xi^j_e \subset I^2 - Z_e$ ($j = 1,2,\ldots$)
of pairwise disjoint disks, each of diameter less than $\eta$,
whose interiors cover the 0-dimensional set  $P_e$.
We split each $\xi^j_e$ into an annulus $A^j_e$ and
disk $E^j_e$, the union of which equals $\xi^j_e$
and intersection of which equals $\partial E^j_e \subset
\partial A^j_e$.  From the construction 
just described we see that each
$\Psi^*_e|I^2 - \cup_j \mathrm{Int }\xi^j_e$ extends over the various $\xi^j_e$ to $F_h \nu^j_e$ on $E^j_e$ 
and as a short homotopy in $\Sigma - \mathrm{Bd }D$
on $A^j_e$ (where $\nu^j_e:E^j_e \to C \cup_\theta C^*$).
Controls on sizes of $\xi^j_e, A^j_e, E^j_e$ and shortness
of the homotopy on $A^j_e$ must be increasingly stringent as $j \to \infty$.
Denote the new maps as $\Psi'_1,\Psi'_2$.

It is a simple matter to perform a general position adjustment in $S^n = C \cup_\theta C^*$
to make the images of the various $E^j_e$ under $\nu_1,\nu_2$
pairwise disjoint, where $\nu_1,\nu_2$ denote the obvious union of maps.  That is, approximate these $\nu_e$
by $\nu'_e$, fixing $\cup_j \partial E^j_e$, so that
$$\nu'_1(\cup_j E^j_1) \cap \nu'_2(\cup_j E^j_2) 
= \emptyset.$$
Then the maps $\Psi_e:I^2 \to \Sigma$
defined as $F_h \nu'_e$ on $\cup_j E^j_e$
and as $\Psi'_e$ elsewhere satisfies
$\Psi_1(I^2) \cap \Psi_2(I^2) \subset \mathrm{Int }D$,
since $F_h$ is 1-1 over $C^*$.
A final general position adjustment
affecting only points sent into Int $D$ 
eliminates all intersections for the
images of $\Psi_1,\Psi_2$.
\end{proof}

\begin{corollary} 
If $(C_1,D_1,h_1)$, $(C_2,D_2, h_2) \in \mathscr{W}_n$ and 
$\theta:\mathrm{Bd}$ $C_1 \to \mathrm{Bd}$ $C_2$ is a sewing such that 
$C_1 \cup_\theta C_2 = S^n$, then 
$D_1 \cup_{h_2 \theta h_1^{-1}} D_2 = S^n$. 
\end{corollary}

\begin{theorem} 
If $(C,D,h) \in \mathscr{W}_n$ and $C \cup_h D = S^n$, $n \geq 5,$ then $D$ contains disjoint 
homotopy taming sets $T$ and $T'$.
\end{theorem}
\begin{proof}
For this proof we regard $D$ as embedded in $S^n$ via its position as a summand in the sewing
space $C \cup_h D = S^n$.  The hypothesis that $(C,D,h) \in \mathscr{W}_n$ means
that there is a retraction $r:S^n \to D$ such that $r^{-1}(\text{Bd }D) = \text{Bd }D$ ---
simply apply the map $f_h:C \to D$ associated with $h$ to the other summand
$C$ of the sewing space.

It suffices to show that any two maps $\mu_1, \mu_2:I^2 \to D$ can be
approximated by maps with disjoint images.  Such maps, regarded as maps
into $S^n$, can be approximated by maps $\mu'_1,\mu'_2:I^2 \to S^n$ with
disjoint images. If these approximations protrude only slightly into 
Int $C$, $r \mu'_1, r \mu'_2$ will be close to $\mu_1$ and $\mu_2$, respectively.
Their images intersect only at points of $\text{Int }D$, so a final 
adjustment over Int $D$  eliminates all intersections, just as in the proof of Theorem 6.3.
\end{proof}

\begin{corollary} 
If $(C,D,h) \in \mathscr{W}_n$  and $C \cup_{\mathbbm{1}} C=S^n$,  $n\geq 5$, then $D \cup_{\mathbbm{1}} D=S^n$.
\end{corollary}

\begin{corollary} 
Suppose $(C,D,h)$ and $(D,C,h^{-1}) \in \mathscr{W}_n$.  Then  $C \cup_h D = S^n$  
if and only if $h$ satisfies the Mismatch Property.
\end{corollary}
\begin{proof}
It is well-known \cite{CannonDaverman 1}\cite{Daverman 5}  that $C \cup_h D = S^n$ if $h$ satisfies the Mismatch
Property. The reverse implication follows from an argument similar to
the one given here for Theorem 6.5, due to the existence of retractions of $S^n = C \cup_h D$  to both $C$ and $D$ that are 1-1 over the boundaries. 
\end{proof}

\begin{corollary} 
Suppose $G$ is a u.s.c. decomposition of $S^n$ such that for any $p\in \pi(H_G)$ and open set $U$ containing $p$ there 
is an open set $V$ such that $p\in V \subset U$ and $\mathrm{Bd}$ $V$ is an $(n-1)$-sphere which misses $\pi(H_G)$; 
suppose $\big(V,\mathrm{Cl}$ $(S^n-V),\mathbbm{1}\big)$ and $\big(\mathrm{Cl}$ $(S^n-V),V,\mathbbm{1}\big)$ are in
$\mathscr{W}_n$. Then $G$ is shrinkable and $S^n/G$ is homeomorphic to $S^n$.
\end{corollary}
\begin{proof}
Apply the main theorem of \cite{Gu 1} and Corollary 6.5.
\end{proof}

\end{document}